\documentclass[12pt]{amsart}
\usepackage[colorlinks=true,pagebackref,hyperindex,citecolor=NavyBlue,linkcolor=Mahogany]{hyperref}
\usepackage[dvipsnames]{xcolor}
\usepackage{verbatim}
\usepackage{amsmath}
\usepackage{amsfonts}
\usepackage{amssymb}
\usepackage{color}
\usepackage[all]{xy}  
\usepackage{enumerate}
\usepackage[top=1in, bottom=1in, left=0.9in, right=0.9in]{geometry}
\usepackage{mathrsfs}

\usepackage{stmaryrd}
\usepackage{verbatim}
\usepackage{tikz}
\usetikzlibrary{shapes.geometric}
\usetikzlibrary{calc,shadows}
\usetikzlibrary{decorations.markings}

\theoremstyle{definition}
\newtheorem{theorem}{Theorem}[section]
\newtheorem{theoremx}{Theorem}

\numberwithin{equation}{section}

\newtheorem{question}[theorem]{Question}
\newtheorem{corollary}[theorem]{Corollary}
\newtheorem{lemma}[theorem]{Lemma}

\newtheorem{notation}[theorem]{Notation}

\newtheorem*{claim*}{Claim}

\theoremstyle{definition}
\newtheorem{remark}[theorem]{Remark}
\newcommand{\e}{\operatorname{e}}

\newtheoremstyle{TheoremNum}
        {8pt}{8pt}              
        {\upshape}                      
        {}                              
        {\bfseries}                     
        {.}                             
        {.5em}                             
        {\thmname{#1}\thmnote{ \bfseries #3}}
  \theoremstyle{TheoremNum}

\newcommand{\m}{\mathfrak{m}}

\newcommand{\ZZ}{\mathbb{Z}}

\newcommand{\reg}{\operatorname{reg}}

\newcommand{\Ht}{\operatorname{ht}}

\newcommand{\HF}{\operatorname{HF}}

\newcommand{\ds}{\displaystyle}

\newcommand{\ov}[1]{\overline{#1}}

\renewcommand{\leq}{\leqslant}
\renewcommand{\geq}{\geqslant}

\newcommand{\LPP}{\operatorname{LPP}}

\newcommand{\sat}{\operatorname{{sat}}}
\newcommand{\kk}{\Bbbk}

\DeclareMathOperator{\Tor}{Tor}

\title[Linearly presented modules and bounds on the regularity of ideals]{Linearly presented modules and bounds on the Castelnuovo-Mumford regularity of ideals}
\author{Giulio Caviglia}
\address{Department of Mathematics, Purdue University, 150 N. University Street, West Lafayette, IN 47907-2067, USA}
\email{gcavigli@purdue.edu}
\author{Alessandro De Stefani}
\address{Dipartimento di Matematica, Universit{\`a} di Genova, Via Dodecaneso 35, 16146 Genova, Italy}
\email{destefani@dima.unige.it}
\thanks{The work of the first named author was partially supported by a grant from the Simons Foundation (41000748, G.C.)}


\begin{document}

\maketitle

\begin{abstract}
We estimate the Castelnuovo-Mumford regularity of ideals in a polynomial ring over a field by studying the regularity of certain modules generated in degree zero and with linear relations. In dimension one, this process gives a new type of upper bounds. By means of recursive techniques this also produces new upper bounds for ideals in any dimension. 
\end{abstract}

\section{Introduction}
Let $S=\kk[x_1,\ldots,x_n]$ be standard graded polynomial ring over a field $\kk$, and $I$ be a homogeneous ideal generated in degree at most $D$. The Castelnuovo-Mumford regularity is a measure of the complexity of the ideal $I$, and has been extensively studied in the literature.

Upper bounds for $\reg(I)$ in terms of $D$ and $n$ have been proved by Galligo \cite{Galligo1, Galligo2} and Giusti \cite{Giusti} over fields of characteristic zero, then extended to all characteristics by Bayer and Mumford \cite{BM} and by the first named author and Sbarra \cite{CS}. They prove that $\reg(I) \leq (2D)^{2^{n-2}}$.

Similar type of results for finitely generated graded modules, or for certain classes of modules or ideals, have been obtained by several authors (for instance, see \cite{CFN,BG,HT}).

A key result of this article is a new type of upper bound for modules of dimension one. This result is inspired by an analogous one obtained by Bruns, Conca and R{\"o}mer in the zero-dimensional case \cite[Remark 3.10 (b)]{BCR}.

For higher dimension, there is a well-known argument that allows to reduce to a lower-dimensional case by expressing the regularity of an ideal in terms of that of its hyperplane sections (see \cite{CS}, and Section \ref{Section regularity}). We optimize this reduction step by introducing correction factors, therefore improving known upper bounds for the regularity of ideals in any dimension. The following is the main result of this article.

\begin{theoremx} (Theorems \ref{THM bound regularity dim at most one} and \ref{THM bound regularity dim at least two}) \label{THMX A} Let $\kk$ be an infinite field, $S=\kk[x_1,\ldots,x_n]$ with the standard grading, and $I$ be a homogeneous ideal generated in degree at most $D$. Let $d=\dim(S/I)$, $h=\Ht(I)$, and $y_1,\ldots,y_d$ be a linear system of parameters in $S/I$. Let $\e(S/I)$ be the multiplicity of $S/I$, and $c$ be the value of Hilbert function of $S/(I+(y_1,\ldots,y_{d-1}))$ in degree $D$ (if $d \leq 1$, by this we simply mean the value of the Hilbert function of $S/I$ in degree $D$). Then 
\begin{align*}
\reg(S/I) \leq \begin{cases} D+c-1 & \text{ if } d \leq 1 \\ 
\left[(D+c-1)\left(D^h-\e(S/I) +1\right)\right]^{2^{d-2}} & \text{ if } d \geq 2.
\end{cases}
\end{align*}
\end{theoremx} 

For zero-dimensional ideals, a sharp upper bound is predicted by the Eisenbud-Green-Harris conjecture \cite{EGH,EGH_CB}. In Corollary \ref{Corollary EGH regularity} we compute this upper bound explicitly, and thanks to Theorem \ref{THMX A} we show show that it holds in some cases (see Remarks \ref{Remark EGH} and \ref{Remark EGH1}).

We would like to point out that, for general upper bounds on $\reg(S/I)$ such as the ones of Theorem \ref{THMX A}, a double exponential behavior in $n$ with base depending on $D$, is inevitable. This is because of a series of well-known examples due to Mayr and Meyer \cite{MM} (see also \cite{BS-MM, Koh}). However, thanks to Theorem \ref{THMX A}, we are able to drastically improve the upper bound $\reg(I) \leq (2D)^{2^{n-2}}$ by dropping a factor of $2^{2^{n-2}}$. In fact, we prove that, if $n \geq 3$ and $I$ is a homogeneous ideal generated in degree at most $D$, then $\reg(I) \leq D^{2^{n-2}}$ (see Corollary \ref{Corollary bound ideals}).

\section{New upper bounds on the regularity of ideals} \label{Section regularity}

Throughout this article, $S$ denotes a standard graded polynomial ring $\kk[x_1,\ldots,x_n]$ over a field $\kk$, and $\m=(x_1,\ldots,x_n)$ its 
graded maximal ideal. Replacing $\kk$ with a field extension will not affect our considerations, therefore we will harmlessly assume throughout that $\kk$ is infinite. Given a finitely generated graded $S$-module $M = \bigoplus_{j \in \ZZ} M_j$, we let $\HF(M;j) = \dim_\kk(M_j)$ denote its Hilbert function in degree $j \in \ZZ$. We let $\reg(M) = \sup\{i+j \mid H^i_\m(M)_j \ne 0\}$ be its Castelnuovo-Mumford regularity, where $H^i_\m(-)$ denotes the $i$-th graded local cohomology module with support in $\m=(x_1,\ldots,x_n)$. Equivalently, $\reg(M) = \sup\{j-i \mid \Tor_i^S(\kk,M)_j \ne 0\}$ (see \cite{EisenbudGoto}).

The next result, even if it involves rather basic techniques, is the crucial point of this article. We estimate the regularity of an ideal by means of the regularity of a related module, which turns out to be linearly presented. This fact, combined with a result of Chardin, Fall and Nagel \cite{CFN}, leads to surprising linear upper bounds for the regularity of ideals of dimension at most one.

\begin{theorem} \label{THM bound regularity dim at most one} 
Let $I$ be a homogeneous ideal such that $\dim(S/I) \leq 1$. Assume that $I$ is generated in degree at most $D$, and let $c=\HF(S/I;D)$. Then $\reg(S/I) \leq D+c-1$.
\end{theorem}
\begin{proof}
If $c=0$, then there is nothing to prove. Moreover, when $n <2$ the claim is straightforward to check, so we will assume that $n \geq 2$.  Let $s_1,\ldots,s_c$ be elements of $S$ whose images modulo $I$ form a $\kk$-basis of $(S/I)_D$. Since $I+(s_1,\ldots,s_c) \subseteq \m^D$, we have an exact sequence 
\[
\xymatrix{
0 \ar[r] & U(-D) \ar[r] & S^{\oplus c}(-D) \ar[rr]^-{[s_1, \ldots, s_c]} && S/I \ar[r] & S/\m^D \ar[r] & 0
}
\]
for some graded submodule $U$ of $F=S^{\oplus c}$. Thus, we obtain that $\reg(S/I) \leq \max\{D+\reg(F/U),\reg(S/\m^D)\} = D+\reg(F/U)$. 

From the exact sequences $\Tor_2^S(\kk,S/\m^D)_j \to \Tor_1^S(\kk,F/U(-D))_j \to \Tor_1^S(\kk,S/I)_j$ of vector spaces, we deduce that $U$ is generated in degree at most one, since both modules on the sides are zero for $j > D+1$. As a consequence, $F/U$ is a module of dimension at most one, generated by $c$ elements of degree zero and related in degrees at most one. If follows from \cite[Theorem 2.1 (ii)]{CFN} applied to the case $l=1$ that $\reg(F/U) \leq c-1$, and the proof is complete.
\end{proof}

For zero-dimensional ideals, the upper bound of Theorem \ref{THM bound regularity dim at most one} had already been proved by Bruns, Conca and R\"omer \cite[Remark 3.10 (b)]{BCR}. However, we point out that in the zero-dimensional case the best possible upper bound in terms of $D$ and $c$ is given by the Eisenbud-Green-Harris Conjecture (for instance, see \cite{EGH,CM}). In order to be more specific, we first recall some definitions and notation. Let $d_1\leq \ldots \leq d_h$ be non negative integers. An ideal $\mathcal L$ is called a $(d_1,\ldots,d_h)$-LPP ideal if $\mathcal L = (x_1^{d_1},\ldots,x_h^{d_h})+L$, where $L$ is a lex-segment ideal. We will make $L$ unique by always choosing the largest possible lex-segment ideal for which the equality holds. The Eisenbud-Green-Harris Conjecture (henceforth, EGH) states that, given any homogeneous ideal $I$ containing a regular sequence of degrees $d_1,\ldots,d_h$, there exists a $(d_1,\ldots,d_h)$-LPP ideal with the same Hilbert function as $I$.

Given a homogeneous ideal $I$ that contains a regular sequence of degrees $d_1,\ldots,d_h$, for $D \geq 0$ we let $\LPP(I;d_1,\ldots,d_h;D)$ be the $(d_1,\ldots,d_h)$-LPP ideal $\mathcal L = (x_1^{d_1},\ldots,x_h^{d_h}) + L$ such that $L$ is either zero or is generated in degree $D$, and $\HF(S/I;D) = \HF(S/\mathcal L;D)$. An equivalent way of stating the EGH Conjecture is to assert that, for any homogeneous ideal $I$ containing a regular sequence of degrees $d_1,\ldots,d_h$, and for any integer $D \geq 0$, one has that $\HF(S/I;D+1) \leq \HF(S/\LPP(I;d_1,\ldots,d_h;D);D+1)$. 

In what follows, whenever $h=n$ and $d_1=\ldots = d_n =D$, we denote $\LPP(I;d_1,\ldots,d_n;D)$ simply by $\LPP(I;D)$. We ask the following question.

\begin{question} \label{Question weak EGH} Let $S=\kk[x_1,\ldots,x_n]$, and $I$ be a zero-dimensional homogeneous ideal generated in degree at most $D$. Is $\reg(S/I) \leq \reg(S/\LPP(I;D))$?
\end{question}

If the EGH Conjecture were to be true, then the Hilbert function of $S/I$ would be point-wise bounded above by the Hilbert function of $S/\LPP(I;D)$, and this would clearly give a positive answer to Question \ref{Question weak EGH}. In this sense, Question \ref{Question weak EGH} can be seen as a weaker version of the EGH Conjecture for zero-dimensional ideals.

In order to be more specific about the upper bound predicted by Question \ref{Question weak EGH}, we explicitly compute the regularity of a zero-dimensional LPP ideal.

\begin{lemma} \label{Lemma regularity LPP} Let $d_1 \leq \ldots \leq d_n$ and $D \geq 2$ be integers, with $D \leq \sum_{i=1}^n (d_i-1)$. Let $\mathcal L = (x_1^{d_1},\ldots,x_n^{d_n}) + L$ be a $(d_1,\ldots,d_n;D)$-LPP ideal, and assume that $\mathcal L \ne (x_1^{d_1},\ldots,x_n^{d_n})$. Set $u=x_a^{t_a} \cdot x_{a+1}^{t_{a+1}} \cdots x_b^{t_b}$, with $t_a \ne 0$, be the smallest monomial with respect to the lex order which has degree $D$, and belongs to $L$. Then $\reg(S/\mathcal L) = t_a-1+\sum_{i=a+1}^n (d_i-1)$. 
\end{lemma} 
\begin{proof}
Let $s = \sum_{i=1}^n (d_i-1)$ and $J=(x_1^{d_1},\ldots,x_n^{d_n})$. Let $j \leq s$, and let $v_j$ be the smallest monomial with respect to the lex order which has degree $j$, and does not belong to $J$. Then $v_j \in \mathcal L$ if and only if $S_j \subseteq \mathcal L_j$. Thus, the regularity of $S/\mathcal L$ is achieved in the highest degree $j$ for which $v_j$ does not belong to $\mathcal L$. It is easy to see that, if $v_j = x_1^{j_1} \cdots x_n^{j_n}$, and $i$ is the smallest index for which $j_i \ne 0$, then $j_r=  d_r-1$ for all $r > i$, and $v_{j-1} = x_i^{j_i-1}x_{i+1}^{d_{i+1}-1} \cdots x_n^{d_n-1}$. Therefore, $v_j  \notin \mathcal L$ if and only if $j_r = 0$ for all $r \in \{1,\ldots,a-1\}$, and $j_a<t_a$. The largest $j$ for which $v_j$ satisfies this condition is $j=t_a-1+\sum_{i=a+1}^n (d_i-1)$. 
\end{proof}

The previous considerations, together with Lemma \ref{Lemma regularity LPP}, lead to the following result.

\begin{corollary} \label{Corollary EGH regularity} 
Let $I$ be a zero-dimensional ideal generated in degree at most $D$. If the EGH Conjecture holds true, then in the same notation as in Lemma \ref{Lemma regularity LPP} we have that 
\[
\reg(S/I) \leq (n-a)(D-1)+t_a-1.
\]
\end{corollary}

\begin{remark} Assuming the validity of the EGH Conjecture, the upper bound of Corollary \ref{Corollary EGH regularity} is clearly sharp, as it is achieved by $S/\LPP(I;D)$.
\end{remark}

\begin{remark} \label{Remark EGH} In the setup of Corollary \ref{Corollary EGH regularity}, if we write $\LPP(I;D) = (x_1^D,\ldots,x_n^D) + L$ and we assume that $c=\HF(S/I;D) < D$, then the smallest monomial of degree $D$ inside $L$ is $x_{n-1}^{c+1}x_n^{D-(c+1)}$. It then follows from Lemma \ref{Lemma regularity LPP} that $\reg(S/\LPP(I;D))  = c+D-1$, which is the upper bound obtained in Theorem \ref{THM bound regularity dim at most one}. It follows that Question \ref{Question weak EGH} has positive answer when $c<D$. 
\end{remark}

If more information about the ideal $I$ is available, then one can aim at sharper upper bounds. For instance, if $I$ contains a regular sequence of degrees $d_1 \leq \ldots \leq d_n$, and the EGH Conjecture with respect to this sequence of degrees holds true, then $\reg(S/I) \leq \reg(S/\LPP(I;d_1,\ldots,d_n;D))$ for all $D \geq 0$. Note that an explicit formula for the right-hand side can be obtained by means of Lemma \ref{Lemma regularity LPP}. One concrete example in which this remark applies is when $d_{i+1} \geq \sum_{j=1}^i (d_i-1)$ for all $i=1,\ldots,n-1$, as proved in \cite[Theorem A]{CDS_EGH}. 

\begin{remark} \label{Remark EGH1} In the setup of Corollary \ref{Corollary EGH regularity}, if $d_1\leq \ldots \leq d_n \leq D$, then one can show that $\HF(\LPP(I;D);j) \leq \HF(\LPP(I;d_1,\ldots,d_n;D);j)$ for all $j \in \ZZ$. Therefore, if the EGH is known to hold for a sequence of degrees $d_1 \leq \ldots \leq d_n \leq D$, then Question \ref{Question weak EGH} has a positive answer for ideals that contain a regular sequence of degrees $d_1,\ldots,d_n$.
\end{remark}

We now return to our original goal of producing improved upper bounds for the Castelnuovo-Mumford regularity of ideals in any dimension. More specifically, we aim at improving a standard recursive argument used in \cite{CS}, which allows to drop the dimension by one at a time, by including several correction terms. Once we reach dimension one, we will finally use the new bounds obtained in Theorem \ref{THM bound regularity dim at most one}. Let us introduce some notation first.

Let $I$ be a homogeneous ideal, and set $R=S/I$. Let $y_1,\ldots,y_t$ be a sequence of linear forms in $R$, and for $i=1,\ldots,t$ let $R^{(i)}$ denote the ring $R/(y_1,\ldots,y_{i})$. For convenience, we let $R^{(0)}=R$.
We recall that a homogeneous non-zero element $y$ is called {\it filter regular} for $R$ if $0:_R y$ has finite length. The sequence $y_1,\ldots,y_t$ is called {\it filter regular} for $R$ if $y_{i+1}$ is filter regular over $R^{(i)}$ for all $i\in \{0,\ldots,t-1\}$. If $\kk$ is infinite, a sufficiently general linear form $\ell \in R^{(i)}$ is a filter regular element for $R^{(i)}$. In particular, a sufficiently general choice of minimal generators of $\m$ forms a filter regular sequence for $R$. Let $y_1,\ldots,y_t$ be a filter regular sequence for $R$. Observe that, in this case, $(I+(y_1,\ldots,y_i))^{\sat} =  (I+(y_1,\ldots,y_i)): \m^\infty = (I+(y_1,\ldots,y_i)):y_{i+1}^\infty$ for all $0 \leq i \leq t-1$, where for $i=0$ we simply mean that $I^{\sat} = I:y_1^\infty$. 

Now assume that $R=S/I$ has Krull dimension $d>1$, and $I$ is generated in degree at most $D$. Let $y_1,\ldots,y_d$ be a filter regular sequence for $R$ consisting of linear forms. If $\ell(-)$ denotes the length of a module, it is shown in the proof of \cite[Theorem 2.4]{CS} that
\begin{equation}
\label{equation CS}
\reg(R) \leq \max\{D,\reg(R^{(1)})\} + \left( \prod_{i=1}^{d-1} \reg(R^{(i)})\right)\cdot  \ell\left(\frac{(I+(y_1,\ldots,y_{d-1}))^{\sat}+(y_d)}{I+(y_1,\ldots,y_d)}\right).
\end{equation}

Observe that 
\begin{align*}
\ell\left(\frac{(I+(y_1,\ldots,y_{d-1}))^{\sat}+(y_d)}{I+(y_1,\ldots,y_d)}\right) & = \ell(R^{(d)}) - \ell\left(\frac{S}{(I+(y_1,\ldots,y_{d-1}))^{\sat}+(y_d)}\right) \\
& = \ell(R^{(d)}) - \e(R). 
\end{align*}

A typical estimate for $\ell(R^{(d)})$ comes from the fact that we can find forms $f_1,\ldots,f_h \in I$ of degree $D$ such that $f_1,\ldots,f_h,y_1,\ldots,y_d$ forms a regular sequence of maximal length in $S$. Thus $\ell(R^{(d)}) \leq \ell(S/(f_1,\ldots,f_h,y_1,\ldots,y_d)) = D^h$. We here provide a slightly more refined estimate for $\ell(R^{(d)})$.

To avoid any confusion, and to be consistent with the notation used in the proof of Theorem \ref{THM bound regularity dim at least two}, in the next lemma we consider a polynomial ring in $h$ variables, rather than in $n$ variables.

\begin{lemma} \label{Lemma bound Artinian reduction} Let $S=\kk[x_1,\ldots,x_h]$, and $I \subseteq S$ be a homogeneous ideal, generated in degree at most $D$, and such that $R=S/I$ is Artinian. Let $c'=\HF(R;D)$. Then 
\[
\displaystyle \ell(R) \leq D^h-{D+h-1 \choose h-1} + c' +h \leq D^h.
\]
\end{lemma} 
\begin{proof} Since $R$ is Artinian and generated in degree at most $D$ the ideal $I$ contains an ideal $J$ generated a regular sequence of $h$ forms of degree $D$. Therefore, we get 
\begin{align*}
\ell(R) & = \ell(S/J) - \ell(I/J) \leq D^h - \HF(I/J;D)  \\
& = D^h - \HF(I;D) + \HF(J;D)  \leq D^h - {D+h-1 \choose h-1}+c' +h.
\end{align*}
Since $c' = \HF(S;D) - \HF(I;D) \leq {D+h-1 \choose h-1} - h$ always holds, the proof is complete. 
\end{proof}

\begin{notation} We let $\Phi(D,c',h) =D^h - {D+h-1 \choose h-1}+c' +h$. 
\end{notation}

We are finally ready to state and prove the second part of our main result.

\begin{theorem} \label{THM bound regularity dim at least two}
Let $S=\kk[x_1,\ldots,x_n]$, and $I\subseteq S$ be a homogeneous ideal of height $h \leq n-2$, generated in degree at most $D$, and let $d=n-h=\dim(S/I) \geq 2$. Let $y_1,\ldots,y_d$ be a system of parameters for $R=S/I$ consisting of linear forms. Let $c = \HF(R/(y_1,\ldots,y_{d-1}); D)$, and $c'=\HF(R/(y_1,\ldots,y_d);D)$. Then
\[
\reg(R) \leq \left((D+c-1) (\Phi(D,c',h)-\e(R)+1) \right)^{2^{d-2}}.
\]
\end{theorem}
\begin{proof}
If $z_1,\ldots,z_d$ is a general choice of a linear system of parameters, then for all $t \in \{1,\ldots,d\}$ and all $j \in \ZZ$ we have that $\HF(R/(z_1,\ldots,z_t);j) \leq \HF(R/(y_1,\ldots,y_t);j)$. Therefore, after possibly replacing $y_1,\ldots,y_d$ with a general choice of a linear system of parameters, we may assume $y_1,\ldots,y_d$ is a filter regular sequence for $R$. For $i=1,\ldots,d$, let $R^{(i)} =R/(y_1,\ldots,y_i)$. Applying (\ref{equation CS}) to $R^{(1)}$ yields $\reg(R^{(1)}) \leq \max\{D,\reg(R^{(2)})\} + A_2$, where we set $A_2= \reg(R^{(2)})\cdots \reg(R^{(d-1)}) \cdot (\ell(R^{(d)})- \e(R))$. Applying (\ref{equation CS}) to $R$, and using the estimate for $\reg(R^{(1)})$ which we have just obtained, we get that
\begin{align*}
\reg(R) &\leq \max\{D,\reg(R^{(2)})\} + A_2 +\left( \max\{D,\reg(R^{(2)})\} + A_2\right) \cdot A_2 \\
& = (1+A_2) \cdot (\max\{D,\reg(R^{(2)})\} + A_2)  \leq \left(\max\{D,\reg(R^{(2)})\} + A_2 \right)^2.
\end{align*}
Repeating the same argument for $R^{(j)}$, with $2 \leq j \leq d-2$, we obtain that
\[
\reg(R)  \leq \left(\max\{D,\reg(R^{(d-1)})\}+A_{d-1}\right)^{2^{d-2}} \leq  \left((D+c-1)+ A_{d-1}\right)^{2^{d-2}},
\]
where $A_{d-1} = \reg(R^{(d-1)}) \cdot (\ell(R^{(d)})-\e(R))$ and the last inequality follows from the fact that $\max\{D,\reg(R^{(d-1)})\}\leq D+c-1$, by Theorem \ref{THM bound regularity dim at most one}. By Lemma \ref{Lemma bound Artinian reduction} we have $\ell(R^{(d)}) \leq \Phi(D,c',h)$, and thus $A_{d-1} 
\leq (D+c-1)(\Phi(D,c',h) -\e(R))$ follows again from Theorem \ref{THM bound regularity dim at most one}. We conclude that
\[
\reg(R) \leq \left((D+c-1) (\Phi(D,c',h)-\e(R)+1) \right)^{2^{d-2}}. \qedhere
\]
\end{proof}

\begin{remark}
By Lemma \ref{Lemma bound Artinian reduction} we have that $\Phi(D,c',h) \leq D^h$, and therefore the bound 
\[
\reg(R) \leq \left((D+c-1) (D^h-\e(R)+1) \right)^{2^{d-2}}
\]
claimed in Theorem \ref{THMX A} now follows from Theorem \ref{THM bound regularity dim at least two}.
\end{remark}

\begin{remark} Let $R=S/I$ be a $d$-dimensional ring and $a=\HF(R;D)$. Let $y_1,\ldots,y_d$ be a linear system of parameters for $S/I$, let $c=\HF(R/(y_1,\ldots,y_{d-1});D)$ and $c'=\HF(R/(y_1,\ldots,y_d);D)$, as in Theorem \ref{THM bound regularity dim at least two}. Consider the $D$-th Macaulay expansion of $a$ (for instance, see \cite{BH}): 
\[
a= {a_1 \choose D} + {a_2 \choose D-1} + \cdots + {a_D \choose 1}.
\]
A repeated application of Green's hyperplane restriction Theorem \cite{Green} yields
\[
c \leq {a_1 - (d-1) \choose D} + {a_2 - (d-1) \choose D-1} + \cdots + {a_D - (d-1) \choose 1}
\]
and 
\[
c' \leq {a_1 - d \choose D} + {a_2 -d \choose D-1} + \cdots + {a_D - d \choose 1}.
\]
Since the right-hand side of the inequality in Theorem \ref{THM bound regularity dim at least two} is increasing in $c$ and $c'$, one can replace their values with the above estimates, and still obtain an upper bound on $\reg(R)$. Alternatively, one can use the Macaulay expansion of $c$ to determine an upper bound for $c'$. If $\ds c= {c_1 \choose D} + {c_2 \choose D-1} + \cdots + {c_D \choose 1}$, then $\ds c' \leq {c_1-1 \choose D} + {c_2-1 \choose D-1} + \cdots + {c_D-1  \choose 1}$.
\end{remark}

Theorem \ref{THM bound regularity dim at least two} allows to drastically improve the upper bound $\reg(I) \leq (2D)^{2^{n-2}}$ which is usually found in the literature. In fact, it can be shown that $\reg(S/I) \leq D^{2^{n-2}}$ holds for all $n \geq 3$ as a consequence of Theorem \ref{THM bound regularity dim at least two}.

\begin{corollary} \label{Corollary bound ideals} Let $S=\kk[x_1,\ldots,x_n]$ with the standard grading, and $I$ be a homogeneous ideal, generated in degree at most $D$. Assume that $n \geq 3$. Then $\reg(I) \leq D^{2^{n-2}}$.
\end{corollary}
\begin{proof}
Let $d=\dim(S/I)$ and $h=n-d = \Ht(I)$. We may assume that $D \geq 2$. If $h=1$, then we can write $I=fI'$ for some $f \in S$ of degree $D' \leq D$ and $I'$ an ideal of height at least two, generated in degree at most $D-D'$. Since $\reg(I) = D'+\reg(I')$, the claim is clear provided we can prove the corollary for ideals of height at least two.

After a change of coordinates, we may assume that $x_n,\ldots,x_{h+1}$ forms a linear system of parameters which form a filter regular sequence for $R=S/I$. Let $\ov{S} = \kk[x_1,\ldots,x_{h+1}]$, so that $R/(x_n,\ldots,x_{h+2})$ can be viewed as a quotient of $\ov{S}$ by a homogeneous ideal $\ov{I}$, still generated in degree at most $D$.

If $\dim(R) \leq 1$, then $\reg(R) \leq (D-1)n$ by \cite[Theorem 3.5]{CFN}. The latter is easily seen to be bounded above by $D^{2^{n-2}}-1$ whenever $n \geq 3$ and $D \geq 2$. Thus, $\reg(I) = \reg(R)+1 \leq D^{2^{n-2}}$.

For the rest of the proof, assume $\dim(R)\geq 2$. Observe that $c=\HF(\ov{S}/\ov{I};D) \leq \HF(\ov{S};D) - \Ht(\ov{I}) = {D+h \choose h} - h$ and, if we let $S'=\ov{S}/(x_{h+1})$ we similarly have that $c' \leq \HF(S';D) - \Ht(I S';D) = {D+h-1 \choose h-1}  - h$. 

Observe that, in the above inequalities, equality holds for $c$ if and only if it holds for $c'$, and this happens if and only if $I$ is generated by a regular sequence of forms of degree $D$. When this is the case we have $\reg(R) =  (D-1)h$, which can be bounded above by $D^{2^{h-1}}-1$ for all $h \geq 2$, as a straightforward calculations shows. Since $h \leq n-2$, the claim follows in this case as well.

We may henceforth assume that $c \leq {D+h \choose h} - h - 1$ and $c' \leq  {D+h-1 \choose h-1}  - h-1$. In this case, we have that $\Phi(D,c',h) = D^h-{D+h-1 \choose h-1} + c' + h \leq D^h-1$. By Theorem \ref{THM bound regularity dim at least two} we get that 
\begin{align*}
\reg(R) &\leq \left[(D+c-1)(\Phi(D,c',h)-\e(R)+1)\right]^{2^{n-h-2}} \\
& \leq \left[\left(D+{D+h \choose h} - h - 2\right)(D^h-\e(R))\right]^{2^{n-h-2}}. 
\end{align*}

A straightforward computation shows that for all $D \geq 2$ and $h \geq 2$ one has $D+{D+h \choose h} - h - 2  \leq D^{h}$. Since $2h \leq 2^h$ for all $h \geq 2$, and because $\e(R) \geq 1$, we finally get
\[
\reg(R) \leq \left[D^h(D^h-\e(R))\right]^{2^{n-h-2}} < D^{2h(2^{n-h-2})} \leq D^{2^{n-2}},
\]
so that $\reg(I) = \reg(R) + 1 \leq D^{2^{n-2}}$ in this case as well.
\end{proof}

We point out that the bound of Corollary \ref{Corollary bound ideals} also follows from \cite[Example 3.6]{CFN}. However, increasing either $D$ or the height $h$ of the ideal $I$, and conducting a careful analysis of the quantities involved in the right-hand side of the inequality in Theorem \ref{THM bound regularity dim at least two}, leads to even more refined estimates than the one above.

\bibliographystyle{alpha}
\bibliography{References}
\end{document}